%% file: LC_EP_v1.tex
\documentclass[onefignum,onetabnum]{siamart190516}
\pdfoutput=1

\input{ex_shared}

\ifpdf
\hypersetup{
  pdftitle={polynomial and rational interpolation: potential, barycentric weights, and Lebesgue constants},
  pdfauthor={Kelong Zhao, and Shuhuang Xiang}
}
\fi

\begin{document}

\maketitle

\begin{abstract}
  In this paper, we focus on barycentric weights and Lebesgue constants for Lagrange interpolation of arbitrary node distributions on \([-1,1]\). The following three main works are included: estimates of upper and lower bounds on the barycentric weights are given in terms of the logarithmic potential function; for interpolation of non-equilibrium potentials, lower bounds with exponentially growing parts of Lebesgue constants are given; and for interpolation consistent with equilibrium potentials, non-exponentially growing upper bounds on their Lebesgue constants are given. Based on the work in this paper, we can discuss the behavior of the Lebesgue constant and the existence of exponential convergence in a unified manner in the framework of potential theory.
\end{abstract}

\begin{keywords}
  Barycentric interpolation, Lebesgue constant, potential theory
\end{keywords}

\begin{AMS}
  65D05, 41A10, 41A20
\end{AMS}

\section{Introduction} \label{sec:int}
For a linear interpolant of $f(x)\in[-1,1]$ 
\begin{equation*}
  I_f^{(n)}(x)=\sum_{k=0}^{n} f_k^{(n)}l_k^{(n)}(x),\quad f_k^{(n)}=f(x_k^{(n)}),\quad -1\le x_0^{(n)}<x_1^{(n)}<\dots<x_n^{(n)}\le1,
\end{equation*}
its Lebesgue constant is defined as 
\begin{equation*}
  \Lambda_n = \max_{x\in [-1,1]} \Lambda_n(x),\quad \Lambda_n(x)=\sum\limits_{k=0}^{n}|l_k^{(n)}(x)|.
\end{equation*}

This constant can be used to measure the amplification of approximation errors due to rounding. 
If the Lebesgue constant grows too rapidly, the method fails to achieve high-order and high-precision approximations. 
Therefore, alongside the development of interpolation theory, the study of Lebesgue constant has a long history 
and remains an important topic, especially for the investigation of Lebesgue constant in specific polynomial 
interpolation \cite{Szego1939,Erdos1961,Brutman1978,Brutman1997} and rational interpolation \cite{Bos2012,Bos2013}.

In these studies, the estimation of Lebesgue constant for interpolation methods on equidistant nodes has long attracted interest \cite{Trefethen1991,Bos2011,Bos2012}. 
One important reason for this is that, in reality, the data and samples will often not be on Chebyshev points. Instead, there is a high probability that they will be on equidistant nodes. 
However, reality is diverse, and it is possible that the data and samples were obtained on other non-isometric nodes. Thus, we can ask a more general question: 
How would the Lebesgue constants behave for polynomials and rational interpolations on arbitrary nodes?

Since Lebesgue constant varies with respect to $n$, these `arbitrary' interpolation nodes cannot become irregular as $n$ tends to infinity; otherwise, there is no point in discussing their Lebesgue constants. 
Therefore, it is reasonable to impose a restriction that the unit measure $\mu_n= 1/(n+1)\sum_{i=0}^{n}\delta_{x_i^{(n)}}$ converges weakly to a unit measure $\mu$, i.e., as given in 
\begin{equation}\label{eq:1.2}
  \lim_{n\to\infty}\int_{-1}^1 g(t) \, \mathrm{d}\mu_{n}(t)=\int_{-1}^1 g(t)w(t) \, \mathrm{d}t
\end{equation}
for arbitrary continuous function $g$ on $[-1,1]$, 
where $w$ is the density function of $\mu$. Then, the `arbitrary' nodes are transformed into nodes with an `arbitrary' density function $w$.

We can define a logarithmic potential function 
\begin{equation}\label{eq:1.3}
  U(x)=\int_{-1}^{1}\log\frac{1}{|x-t|}w(t)\, \mathrm{d}t+\phi(x)
\end{equation}
where $\phi$ is an external field. Similarly one can define a discrete potential function for a rational interpolant $r_{n,m}$ at $\{x_i^{(n)}\}_{i=0}^n$
\begin{equation}\label{eq:1.4}
  U_{n}(x)=\frac{1}{n+1}\sum_{i=0}^{n}\log\frac{1}{|x-x_{i}^{(n)}|}+\phi_n(x),\quad \phi_n(x)=\frac{1}{n+1}\sum_{j=1}^{m}\log{|x-p_{j}^{(n)}|}.
\end{equation}
where $\{p_j^{(n)}\}_{j=1}^m$ are all poles of $r_{n,m}$. When the external field is $0$, it corresponds to polynomial interpolation.

If the discrete potential \cref{eq:1.4} tends to the steady-state potential function \cref{eq:1.3} as $n$ grows, then a clear pattern emerges in 
the rate of growth of the Lebesgue constant concerning the \cref{eq:1.3}. We illustrate this with a simple test.
\begin{itemize}
  \item The first test involves polynomial interpolation on equidistant nodes, corresponding to a density function of $w(t) = 1/2$ and a potential function given by 
  \begin{equation*}
    U_a(x)=1-\frac{1}{2}{((x+1)\log(x+1)-(x-1)\log(x-1))}.
  \end{equation*}
  \item The second test pertains to rational interpolation $r_{n,n}$ using the Chebyshev points. In this case, the corresponding density function is $w(t) = 1/(\pi\sqrt{1-x^2})$, 
  the $n$ poles are located at $\pm0.5\mathrm{i}$, and the potential function is given by 
  \begin{equation*}
    U_b(x)=\log{2}+\phi(x),\quad \phi(x)=\frac{1}{2}\log|x-0.5\mathrm{i}|+\frac{1}{2}\log|x+0.5\mathrm{i}|.
  \end{equation*}
  \item The third test deals with rational interpolation $r_{n,n}$ on equidistant nodes. 
  The corresponding density function is $w(t) = 1/2$, the $n$ poles are located at $\pm0.5\mathrm{i}$, 
  and the potential function is given by 
  \begin{equation*}
    U_c(x)=1-\frac{1}{2}{((x+1)\log(x+1)-(x-1)\log(x-1))}+\phi(x).
  \end{equation*}
\end{itemize}

Since the nodes $\{x_i^{(n)}\}_{i=0}^n$ and poles $\{p_j^{(n)}\}_{j=1}^m$ of these interpolations are known, we use barycentric interpolation \cite{Berrut1997} 
\begin{equation} \label{eq:1.1c}
  r_{n,m}(x)=\sum\limits_{k=0}^{n}\dfrac{w_kf_k}{x-x_k}\Big/{\sum\limits_{k=0}^{n}\dfrac{w_k}{x-x_k}},
  \,\,  w_k = C \dfrac{\prod_{j=1}^m(x_k-p_j)}{\prod_{i=0,i\ne k}^n(x_k-x_i)},\,\, C\not=0.
\end{equation}
for this purpose. Similar to Lebesgue constant, the quotient of the largest barycentric weight by the smallest in absolute values, 
as given in $\max|w_k|/\min|w_k|$, gives an intuitive estimation on the quality of the interpolation method. 
Based on the upper plot of \cref{fig:1.1b}, it can be observed that the maximum ratio of all three interpolations 
grows exponentially as $[\exp(d)]^n$, where $d$ happens to be the difference in the steady-state potentials, 
as we will demonstrate in \cref{sec:3}.

\begin{figure}[htbp]
  \centering
  \subfloat[]{
    \label{fig:1.1a}
    \includegraphics[width=0.38\linewidth]{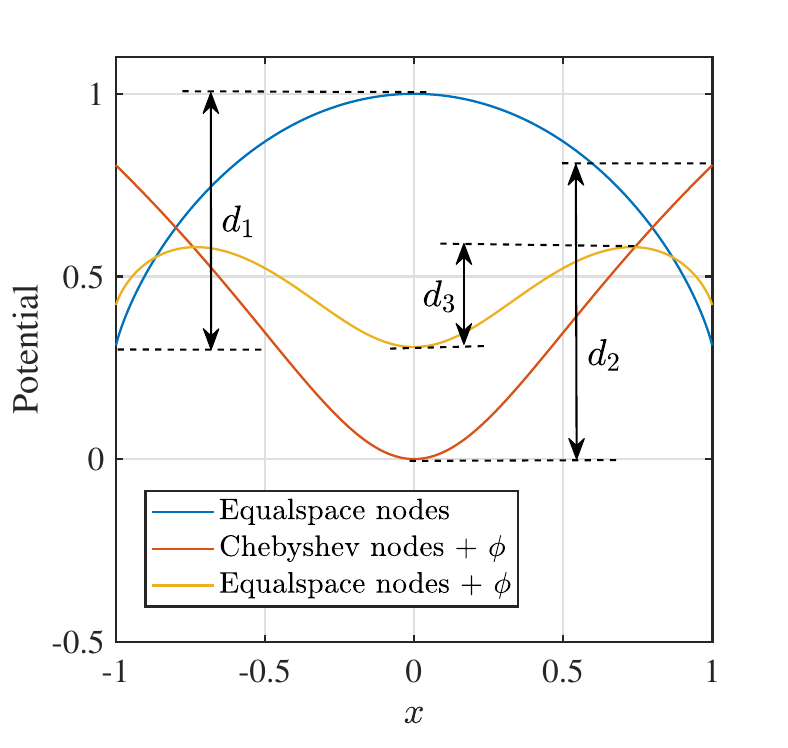}
  }
  \subfloat[]{
    \label{fig:1.1b}
    \includegraphics[width=0.56\linewidth]{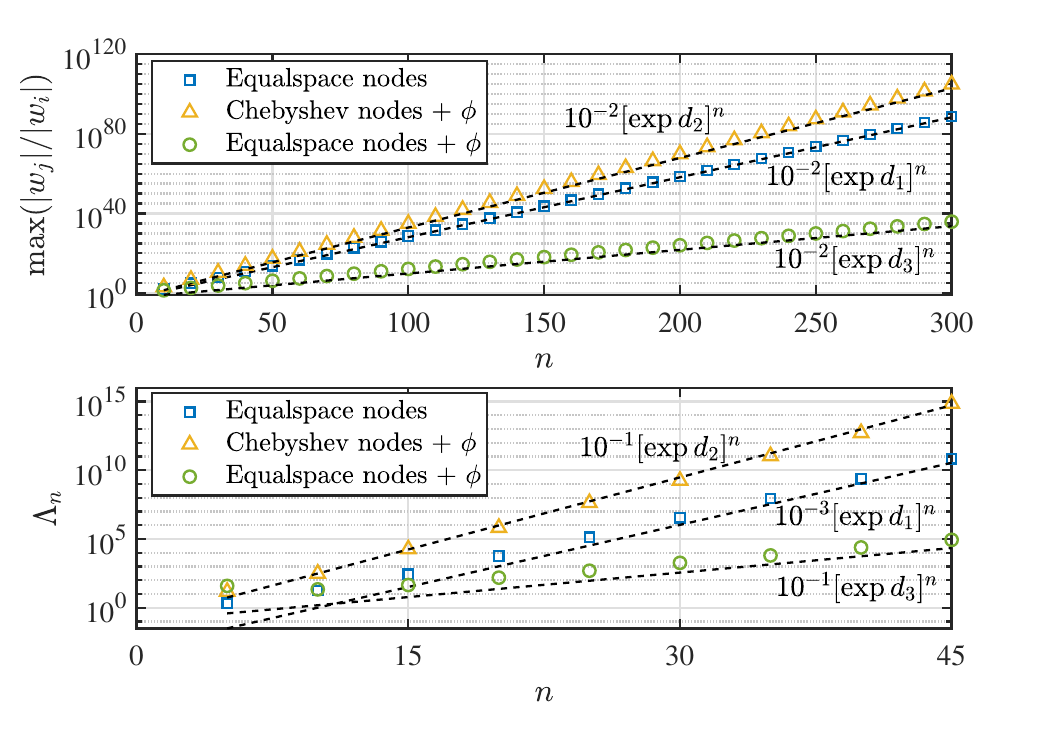}
  }
  \caption{Three potential functions, denoted as $U_a,\,U_b,\,U_c$, along with the differences between their respective maxima and minima: $d_1$, $d_2$, $d_3$. 
  (b): The maximum value of the ratio of absolute values of barycentric weights (top) and the Lebesgue constant (bottom) for the three interpolants. The reference growth rate is $O([\exp d_k]^n)$ ($k=1,2,3$).
  }
  \label{fig:1.1}
\end{figure}

Using \cref{sec:3} as a foundation, we will show in \cref{sec:4} that there exists a lower bound for the Lebesgue 
constants with a similar exponential growth factor $[\exp(d)]^n$. Thus, these Lebesgue constants also exhibit exponential 
growth (as shown in the lower plot of \cref{fig:1.1b}) when the potential function is non-constant on the interval $[-1, 1]$.

A new question arises: If the potential function \cref{eq:1.3} is constant on the interval $[-1, 1]$, how rapidly does the Lebesgue constant grow? 
In \cref{sec:5}, we will provide an upper bound on Lebesgue constant in this scenario. This upper bound includes a logarithmic growth term 
as well as an unknown growth term $\exp(2(n+1)\delta_n^{\pm})$, where $\delta_n^{\pm}$ characterizes the rate at which the discrete potential function \cref{eq:1.4} 
converges to the potential function \cref{eq:1.3}. When $\delta_n^{-} = \mathcal{O}(1/n)$, the Lebesgue constant has an upper bound with logarithmic growth. 

One advantage of polynomial and rational interpolation is the exponential rate of convergence for analytic functions. 
The potential-theoretic explanation of this phenomenon is beautiful and classical. In \cref{sec:6}, we will discuss 
both exponential convergence and Lebesgue constant in terms of potential theory. A sufficient condition for fast and 
stable rational interpolation is given without proof, based on some rational interpolations \cite{Floater2007,Guttel2012,Zhao2023}.

\section{Preliminaries}\label{sec:2}
In this paper, our only requirement for interpolation nodes is that their discrete measures $\mu_n$ converge weakly to some positive measure $\mu$. However, obtaining an intuition for these discrete points from Eq. \cref{eq:1.2} is challenging. 
A definition of the discrete points and the density function was provided in our previous paper , and here it is restricted to the interval $[-1, 1]$:
A definition of a discrete point "obeying" the density function \(w\) was given in our previous paper \cite{Zhao2023}. 
\cref{def:1.1} is sufficient for the weak star convergence of the measure $\mu$. In fact, it is easy to show that \cref{def:1.1} is also necessary for the weak star convergence of the measure $\mu$ on \([-1,1]\). 
The definition restricted to \([-1,1]\) is given here:
\begin{definition}\label{def:1.1}
  A family of point sets $\{\{x_k^{(n)}\}_{k=0}^n:n=1,2,\cdots\}\subset[-1,1]$ obeys the density function $w(t)>0$ for all $t\in[-1,1]$ if, for any segment $[a,b]\subseteq[-1,1]$, the family of point sets satisfies:
  \begin{equation}\label{eq:3a}
  \lim_{n\to\infty}\frac{n_{[a,b]}}{n+1}=\int_a^b w(t)\,\mathrm{d}t,
  \end{equation}
  where $n_{[a,b]}$ denotes the number of points on $[a,b]$.
\end{definition}

In addition to the constraints on the nodes, we also have a requirement on the poles that the external field \(\phi_n\) 
generated by the poles converges consistently to \(\phi\) on \([-1,1]\). It may seem difficult to find poles to satisfy 
this condition, but in fact, we do not need poles to construct rational interpolants with external fields \(\phi_n\).
Recalling Eq. \cref{eq:1.1c}, we have  
\begin{equation*}
  |w_k|=C \dfrac{\prod_{j=1}^m|x_k-p_j|}{\prod_{i=0,i\ne k}^n|x_k-x_i|}=C \dfrac{\exp[(n+1)\phi_n(x_k)]}{\prod_{i=0,i\ne k}^n|x_k-x_i|},\,\, C\not=0.
\end{equation*}
Since the barycentric weights on the real axis alternate between positive and negative values, the weights can be expressed as 
\begin{equation}\label{eq:2.1}
  w_k=C (-1)^k{\exp[(n+1)\phi_n(x_k)]}\Big/{\prod_{i=0,i\ne k}^n|x_k-x_i|},\,\, C\not=0.
\end{equation}



\section{Barycentric weights and potential}\label{sec:3}
For any rational function $r_{n,m}$ (where $m \le n$), it can be expressed in the form of barycentric, as described 
in Eq. \cref{eq:1.1c} \cite{Berrut1997}. In this representation, the coefficient $C$ of the barycentric weights can be any nonzero constant. 
The quotient of the largest barycentric weight by the smallest in absolute value is then independent of $C$ and 
the rate of growth of this ratio reflects, to some extent, the quality of an interpolant \cite{Klein2012}.

For polynomial interpolation, there exists a lower bound for the Lebesgue constant, as given by 
\begin{equation*}
  \Lambda_n\ge\frac{1}{2n^2}{\max\limits_{0\le k\le n }|w_k|}\Big/{\min\limits_{0\le k\le n }|w_k|}.
\end{equation*}
This lower bound is one of the reasons for studying the barycentric weights. In the case of more general rational 
interpolation, we can establish a connection between the Lebesgue constant and the barycentric weights as well, 
but this connection relies on the potential function as a bridge, which we will describe in the next section.

We assume that the constant $C$ of the barycentric weights in Eq. \cref{eq:1.1c} is set to $1$, and we express 
the absolute value of barycentric weight in terms of the discrete potential function \cref{eq:1.4}
\begin{equation}\label{eq:3_wk}
  |w_k^{(n)}| = \dfrac{\prod_{j=1}^m|x_k^{(n)}-p_j^{(n)}|}{\prod_{i=0,i\ne k}^n|x_k^{(n)}-x_i^{(n)}|}= \exp[(n+1)u_k^{(n)}(x_k^{(n)})],
\end{equation}
where 
\[
  u_k^{(n)}(x)=U_n(x)+\frac{\log|x-x_k^{(n)}|}{n+1}.
\]
Therefore, this part will be developed based on the properties of the discrete potential.

\begin{lemma}\label{le:3.1}
  Assume a class of rational interpolants $r_{n,m}$ with nodes $\{x_i^{(n)}\}_{i=0}^n$ in the interval $[-1,1]$ obeying 
  a positive density function $w$, and whose poles $\{p_j^{(n)}\}_{j=1}^m$ lie outside $[-1,1]$, generating an external 
  field $\phi_n$ that converges to $\phi$. If $\phi_n$ is twice differentiable on $[-1,1]$ and its second derivative has 
  a lower bound $M$, then there exists $N > 0$ such that for $n > N$, the discrete potential $U_{n}$ of the rational 
  interpolants is convex on $[-1,1]-\{x_i^{(n)}\}_{i=0}^n$.
\end{lemma}

\begin{proof}
  Recall the definition of discrete potential
  \[
    U_{n}(x)=\frac{1}{n+1}\sum_{i=0}^{n}\log\frac{1}{|x-x_{i}^{(n)}|}+\phi_n(x),
  \]
  where 
  \[
    \phi_n(x)=\frac{1}{n+1}\sum_{j=1}^{m}\log{|x-p_{j}^{(n)}|}, 
  \]
  we have
  \begin{align}\label{eq:u2}
    U_{n}''(x)&=\frac{1}{n+1}\sum_{i=0}^{n}\frac{1}{|x-x_{i}^{(n)}|^2}+\phi_n''(x),\\
    &=\int_{-1}^{1}\frac{1}{|x-t|^2}\, \mathrm{d}\mu_{n}(t)+\phi_n''(x),
  \end{align}
  where $\mu_n=\sum_{i=0}^{n}\delta_{x_i^{(n)}}/(n+1)$ and $x\notin \{x_i^{(n)}\}_{i=0}^n$.

  Let a function
  \begin{equation}
    F_{M_1}(x,t)=\min\{M_1,1/(x-t)^2\}
  \end{equation}
  where $M_1$ such that $\int_{-1}^1 F_{M_1}(x,t)\,\mathrm{d}t>-M$ for all $x\in[-1,1]$. 
  Then there exists $N_1$ such that $\int_{-1}^1 F_{M_1}(x,t)\,\mathrm{d}\mu_{n}(t)>-M$ for all $n > N_1$.
  Therefore, we have
  \[
    U_{n}''(x)>\int_{-1}^1 F_{M_1}(x,t)\,\mathrm{d}\mu_{n}(t)+M>0.
  \]
  for all $n>N_1$ and $x\in[-1,1]-\{x_i^{(n)}\}_{i=0}^n$.

\end{proof}

\begin{remark}\label{re}
  Clearly, the conclusion of \cref{le:3.1} still holds if we remove one point from the $n+1$ interpolation nodes $\{x_i^{(n)}\}_{i=0}^n$.
\end{remark}

Since for sufficiently large $n$, the discrete potential $U_n$ is a convex function on each $(x_{i-1}^{(n)}, x_i^{(n)})$. Since the value of the discrete potential on the nodes is $+\infty$, 
then there exists a unique minima of $U_n$ in each $(x_{i-1}^{(n)}, x_i^{(n)})$. Therefore we give a new definition: 

\begin{definition}\label{def:3.2}
  For sufficiently large $n$,  the discrete potential $U_{n}$ is convex on each $(x_{i-1}^{(n)}, x_i^{(n)})$, so there exists a unique $\{\zeta_i^{(n)}\}_{i=1}^n\in[-1,1]$ such that $U_{n}'(\zeta_i^{(n)})=0$ and $\zeta_i^{(n)}\in(x_{i-1}^{(n)}, x_i^{(n)})$. We call $\{\zeta_i^{(n)}\}_{i=1}^n$ the inter-potential points of $\{x_i^{(n)}\}_{i=0}^n$.
\end{definition}

\vspace{0.5 em}

These inter-potential points have the following property:
\begin{lemma}\label{le:3.3}
  If the rational interpolation's nodes $\{x_i^{(n)}\}_{i=0}^n\in[-1,1]$ satisfies
\[
  a_1 n^{-b_1}\le (x_{i}^{(n)}-x_{i-1}^{(n)})\le a_2 n^{-b_2},
\]
the external field $\phi_n$ satisfies 
\[
|\phi_n'(x)|<M_2\log{n},\quad \forall x\in(-1+a_1n^{-b_1-1},1-a_1n^{-b_1-1}).
\]
Then for sufficiently large $n$, there is
\[
\zeta_i^{(n)}\in(x_{i-1}^{(n)}+a_1n^{-b_1-1}, x_i^{(n)}- a_1n^{-b_1-1}).
\]
\end{lemma}

\begin{proof}
By \cref{def:3.2}, we have $U_n'(\zeta_i^{(n)})=0$, where 
\begin{equation}
  U_n'(x)=\frac{1}{n+1}\sum_{i=0}^{n}\frac{-1}{x-x_{i}^{(n)}}+\phi_n'(x).
\end{equation}  
Notice that $U_n''(x)>0$ for $x\in[-1,1]-\{x_i^{(n)}\}_{i=0}^n$, so we only need to prove
\begin{equation}
  U_n'(x_{i-1}^{(n)}+a_1n^{-b_1-1})<0,\quad U_n'(x_{i}^{(n)}-a_1n^{-b_1-1})>0.
\end{equation}

\begin{itemize}
  \item When $x=x_{i-1}^{(n)}+a_1n^{-b_1-1}$:
  \begin{align}
    U_n'(x)&=\frac{1}{n+1}\sum_{k=0}^{n}\frac{-1}{x_{i-1}-x_k+a_1n^{-b_1-1}}+\phi'(x)\notag \\
    &\le \frac{-n^{b_1+1}}{a_1(n+1)}+\frac{1}{n+1}\sum_{k=i}^{n}\frac{n^{b_1}}{(k-i+1)a_1-a_1n^{-1}}+M_2\log{n} \\
    &=-\frac{n^{b_1}}{a_1(n+1)}(n-\sum_{k=1}^{n-i+1}\frac{1}{k-n^{-1}})+M_2\log{n}.
  \end{align}
  It is easy to prove $b_1\ge 1$, then  $U_n'(x)<0$ for sufficiently large $n$.
  
  \item When $x=x_{i}^{(n)}-a_1n^{-b_1-1}$:
  \begin{align}
    U_n'(x)&=\frac{1}{n+1}\sum_{k=0}^{n}\frac{-1}{x_{i-1}-x_k-a_1n^{-b_1-1}}+\phi'(x)\notag \\
    &\ge \frac{n^{b_1+1}}{a_1(n+1)}-\frac{1}{n+1}\sum_{k=0}^{i-1}\frac{n^{b_1}}{(i-k)a_1-a_1n^{-1}}-M_2\log{n} \\
    &=\frac{n^{b_1}}{a_1(n+1)}(n-\sum_{k=1}^{i}\frac{1}{k-n^{-1}})-M_2\log{n}.
  \end{align}
  So $U_n'(x)>0$ for sufficiently large $n$.
\end{itemize}
\end{proof}

Based on the above lemmas, we can give upper and lower bound estimates for the barycentric weights as follows: 

\begin{theorem}\label{th:3.4}
Suppose a family of rational interpolants $r_{n,m}$ on $[-1,1]$ with nodes $\{x_i^{(n)}\}_{i=0}^n$ and its poles $\{p_j^{(n)}\}_{j=1}^m\notin[-1,1]$, the node obeys the density function $w$, the external field $\phi_n$ converges to $\phi$ and its inter-potential point is $\{\zeta_i^{(n)}\}_{i=1}^n$.
If
\begin{align*}
  &\mathrm{(i)}\,\,\,  a_1 n^{-b_1}\le (x_{i}^{(n)}-x_{i-1}^{(n)})\le a_2 n^{-b_2};\\
  &\mathrm{(ii)}\,\,  |\phi_n'(x)|<M_2\log{n},\quad \forall x\in(-1+a_1n^{-b_1-1},1-a_1n^{-b_1-1});\\
  &\mathrm{(iii)}\, -\delta_n^{-}\le U_{n}(\zeta_i^{(n)})-U(\zeta_i^{(n)}) \le \delta_n^{+},\quad \forall\, 1\le i\le n,
\end{align*}
then
\[
\frac{a_1}{en^{b_1+1}}[e^{U(\zeta_{k_i}^{(n)})}]^{n+1}e^{-(n+1)\delta_n^-}<|w_i|<\frac{a_2}{n^{b_2}}[e^{U(\zeta_{k_i}^{(n)})}]^{n+1}e^{(n+1)\delta_n^+}
\]
where $u_i^{(n)}(x)=U_n(x)+\frac{1}{n+1}\log{|x-x_i|}$ and $\zeta_{k_i}=\arg\max\{u_i^{(n)}(\zeta_i^{(n)}),u_i^{(n)}(\zeta_{i+1}^{(n)})\}$ 
for $0<i<n$ $(k_0=1,\, k_n=n-1)$.
\end{theorem}

\begin{proof}
  From \cref{re}, we have $u_i^{(n)}$ is convex on $(x_{i-1}^{(n)},x_{i+1}^{(n)})$ if $0<i<n$. Then 
  \begin{align}\label{eq:3.4}
    u_i^{(n)}(\zeta_{k_i})+(u_i^{(n)})'(\zeta_{k_i}^{(n)})(-\zeta_{k_i}^{(n)}&+x_i)\le u_i^{(n)}(x_i)\le u_i^{(n)}(\zeta_{k_i}^{(n)})\\
    U_{n}(\zeta_{k_i}^{(n)})+\frac{\log|-\zeta_{k_i}^{(n)}-x_i|-1}{n+1}\le &u_i^{(n)}(x_i)\le U_{n}(\zeta_{k_i}^{(n)})+\frac{\log|-\zeta_{k_i}^{(n)}-x_i|}{n+1}.\notag
  \end{align}
  Considering the barycentric weights
  \begin{equation*}
    |w_k^{(n)}| = \dfrac{\prod_{j=1}^m|x_k-p_j|}{\prod_{i=0,i\ne k}^n|x_k-x_i|}= \exp[(n+1)u_k^{(n)}(x_k)],
  \end{equation*}
  we have 
  \begin{equation*}
    U_{n}(\zeta_{k_i}^{(n)})+\frac{\log|-\zeta_{k_i}^{(n)}-x_i|-1}{n+1}\le\frac{\log|w_i^{(n)}|}{n+1}\le U_{n}(\zeta_{k_i}^{(n)})+\frac{\log|-\zeta_{k_i}^{(n)}-x_i|}{n+1}.
  \end{equation*}
  From \cref{le:3.3}, it is easy to prove $a_1n^{-b_1-1}<|\zeta_{k_i}^{(n)}-x_i|<a_2 n^{-b_2}$. Then we have
  \begin{equation*}
    U_{n}(\zeta_{k_i}^{(n)})+\frac{\log{a_1}-\log{n^{b_1+1}}-1}{n+1}<\frac{\log|w_i^{(n)}|}{n+1}< U_{n}(\zeta_{k_i}^{(n)})+\frac{\log{a_2}-\log{n^{b_2}}}{n+1}.
  \end{equation*}
  Therefore the barycentric weights has
  \begin{equation}\label{eq:3.5}
    \frac{a_1e^{-(n+1)\delta_n^-}}{en^{b_1+1}}[e^{U(\zeta_{k_i}^{(n)})}]^{n+1}<|w_i^{(n)}|<\frac{a_2e^{(n+1)\delta_n^+}}{n^{b_2}}[e^{U(\zeta_{k_i}^{(n)})}]^{n+1}.
  \end{equation}
  for a sufficiently large $n$ and $0<i<n$.

  When \(i=0\), \(u_i^{(n)}\) is convex on \([-1, x_1^{(n)})\). It is easy to show that \(u_0^{(n)}\) is monotonically 
  increasing on \([-1, x_1^{(n)})\) when \(n\) is large enough. Thus, \(u_0^{(n)}\) satisfies Eq. \cref{eq:3.4}. 
  Similarly, it follows that \(u_n^{(n)}\) also satisfies Eq. \cref{eq:3.4}. Then, Eq. \cref{eq:3.5} also holds 
  when \(i=0\) or \(n\).
\end{proof}

If $U$ is constant on the interval $[-1,1]$, then for all absolute values of the barycentric weights $|w_i|$, the upper and lower bounds share the same exponential term. This suggests that we can adjust the value of $C$ in Eq. \cref{eq:1.1c} to eliminate this common exponential factor. Consequently, we can derive simplified weights that do not exhibit exponential growth with $n$.

If $U$ is continuous on the interval $[-1,1]$ but not constant, we can let  
\begin{equation*}
  \underline{x} \in \arg\min_{x\in[-1,1]}(U(x)),\quad \overline{x} \in \arg\max_{x\in[-1,1]}(U(x)).
\end{equation*}
Given the existence of a consistent upper bound on the distance between neighboring nodes, denoted by $a_2n^{-b_2}$, 
we can always find two sequences of nodes $\{x_{i(n)}^{(n)}\}_{n=1}^\infty$ and $\{x_{j(n)}^{(n)}\}_{n=1}^\infty$ such that 
\begin{equation*}
\lim_{n\to\infty}x_{i(n)}^{(n)}=\underline{x},\quad \lim_{n\to\infty}x_{j(n)}^{(n)}=\overline{x}.
\end{equation*}
holds. This leads to the following corollary:

\begin{corollary}
  For sufficiently large $n$, 
  \begin{equation*}
    \frac{a_1n^{b_2-b_1-1}}{ea_2}e^{-(n+1)(\delta_n^-+\delta_n^+)}\rho^{n+1}\le\dfrac{\max\limits_{0\le k\le n}|w_k^{(n)}|}{\min\limits_{0\le k\le n}|w_k^{(n)}|}\le \frac{ea_2}{a_1n^{b_2-b_1-1}}e^{(n+1)(\delta_n^-+\delta_n^+)}\rho^{n+1}
  \end{equation*}
  holds, where $\rho=e^{U(\overline{x})-U(\underline{x})}$.
\end{corollary}

\section{Lebesgue constants for exponential growth}\label{sec:4}
The Lebesgue function can also be expressed in terms of the barycentric weights, as shown in 
\begin{equation*}
  \Lambda_n(x)=\sum_{k=0}^{n}|w_k|\prod_{i=0,i\ne k}^{n}|x-x_i|\Big/\prod_{j=1}^{m}|x-p_j|,
\end{equation*}
where the barycentric weights are provided in Eq. \cref{eq:3_wk}. Consequently, the Lebesgue function can be 
inscribed by the potential function.
Considering the definition of Lebesgue constant, it is easy to obtain the following rough lower bound on the potential function.

\begin{theorem}\label{th:4.1}
  With the same premises as \cref{th:3.4}, if $U$ is a continuous function on $[-1,1]$, then for a sufficiently large $n$, the Lebesgue constants $\Lambda_n$ have 
  \[
    \Lambda_n\ge\frac{a_1e^{-(n+1)(\delta_n^++\delta_n^-)}}{2en^{b_1+1}} \rho^{(n+1)},\quad\rho=e^{U(\overline{x})-U(\underline{x})}.
  \]
\end{theorem}
\begin{proof}
  First we represent the Lebesgue function using discrete potentials $U_n$ and barycentric weights $w_k$:
  \begin{align*}
    \Lambda_n(x)=\sum_{k=0}^{n}|R_k^{(n)}(x)| &= \sum_{k=0}^{n}|w_k|\prod_{i=0,i\ne k}^{n}|x-x_i|\Big/\prod_{j=1}^{m}|x-p_j|\\
    &=\sum_{k=0}^{n}[e^{-U_n(x)}]^{n+1}\frac{|w_k|}{|x-x_k|},
  \end{align*}
  where $R_k^{(n)}(x_k)=1$ and $R_k^{(n)}(x_i)=0$ for $0\le i \le n\,(i\ne k)$. 
  Considering \cref{th:3.4}, we have
  \begin{align*}
    \Lambda_n(\zeta_j^{(n)}) >|R_i^{(n)}(\zeta_j^{(n)})|&=[e^{-U_n(\zeta_j^{(n)})}]^{n+1}\frac{|w_k|}{|\zeta_j^{(n)}-x_k|}\\
    & >[e^{-U(\zeta_j^{(n)})}]^{n+1}e^{-(n+1)\delta_n^+}\frac{a_1}{en^{b_1+1}}[e^{U(\zeta_{k_i}^{(n)})}]^{n+1}e^{-(n+1)\delta_n^-}/2\\
    & = \frac{a_1e^{-(n+1)(\delta_n^++\delta_n^-)}}{2en^{b_1+1}} [e^{U(\zeta_{k_i}^{(n)})-U(\zeta_j^{(n)})}]^{n+1} 
  \end{align*}
  for all $0\le i,\,j\le n$.
  
  Since \cref{def:3.2}, we can find $\{\zeta_{\alpha_n}^{(n)}\}_{n=1}^\infty$ $(1\le\alpha_n\le n)$ and $\{\zeta_{\beta_n}^{(n)}\}_{n=1}^\infty$ $(1\le\beta_n\le n)$
  such that 
  \begin{equation*}
    \lim_{n\to\infty} \zeta_{\alpha_n}^{(n)} = \overline{x},\quad \lim_{n\to\infty} \zeta_{\beta_n}^{(n)} = \underline{x}.
  \end{equation*}
  Therefore, the Lebesgue constants $\Lambda_n$ have 
  \[
    \Lambda_n\ge\frac{a_1e^{-(n+1)(\delta_n^++\delta_n^-)}}{2en^{b_1+1}} \rho^{(n+1)},\quad\rho=e^{U(\overline{x})-U(\underline{x})}.
  \]
  for a sufficiently large $n$.
\end{proof}

The result of \cref{th:4.1} is rather crude. This limitation arises from our use of only one basis function 
whose absolute value exhibits exponential growth. The Lebesgue function, however, is the sum of the absolute values 
of $n+1$ basis functions. Despite this, our choice of a basis function with the fastest exponential growth rate 
allows us to inscribe the exponential growth rate of the Lebesgue constant by the difference of the potential functions, 
precisely what we aim to demonstrate. The coefficients and algebraic growth part of \cref{th:4.1} are not our primary focus.

Referring to the proof of \cref{th:4.1}, we can determine the exponential growth rate of the Lebesgue function at 
any point $\hat{x}\in[-1,1]$. Assuming that $\hat{x}$ is not always an interpolation node, the exponential part of 
the growth rate of $R_i^{(n)}(\hat{x})$ with respect to $n$ does not exceed $[\exp(U(\hat{x})-U(\underline{x}))]^n$, 
where $x_i^{(n)}$ converges to $\underline{x}$. Since $U(\underline{x})$ is a minimum, the exponential part of the growth rate 
of $\Lambda_n(\hat{x})$ likewise does not exceed $[\exp(U(\hat{x})-U(\underline{x}))]^n$.

To verify the above conclusions, we first consider an example of polynomial interpolation (EXAMPLE 1). The density 
function of its nodes is given by 
\begin{equation*}
  w(t)=\frac{\exp(-t^2)}{\sqrt{\pi}\mathrm{Erf}(1)},
\end{equation*}
where $\mathrm{Erf}$ is Gauss error function \cite{Abramowitz1964}.
This density function is the result of a Gaussian distribution restricted to $[-1,1]$ and normalized. 
\cref{fig:4.1a} illustrates the corresponding Lebesgue function for this polynomial interpolation 
at different values of $n$. It can be observed that the Lebesgue function grows at varying rates at different points.
\begin{figure}[htbp]
  \centering
  \subfloat[]{
    \label{fig:4.1a}
    \includegraphics[width=0.47\linewidth]{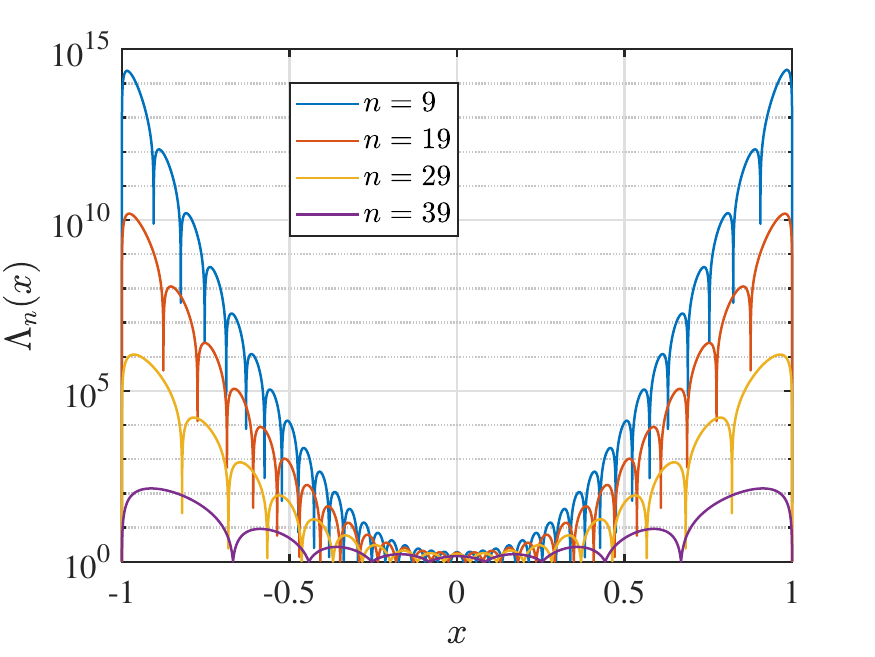}
  }
  \subfloat[]{
    \label{fig:4.1b}
    \includegraphics[width=0.47\linewidth]{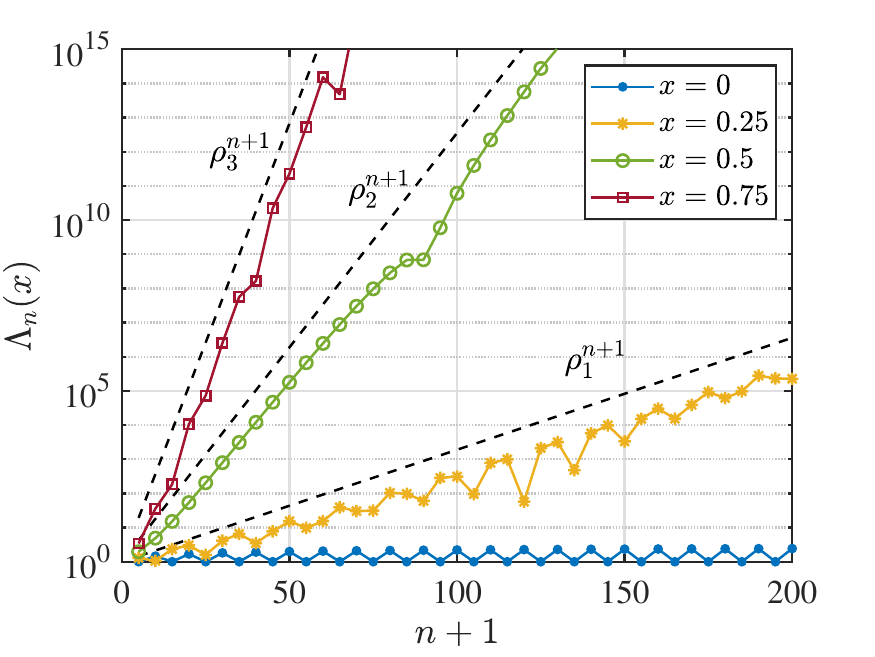}
  }
  \caption{(a): Lebesgue functions of EXAMPLE 1 at different $n$. (b): Values of Lebesgue functions of EXAMPLE 1 at different points.}
  \label{fig:4.1}
\end{figure}

\cref{fig:4.1b} then presents the values of the Lebesgue function at $0$, $0.25$, $0.5$, and $0.75$, 
along with the reference growth rate we have provided. In the figure, $\rho_1=\exp(U(0)-U(0.25))$, $\rho_2=\exp(U(0)-U(0.5))$, and $\rho_3=\exp(U(0)-U(0.75))$, 
since $U(0)$ is the maximum value of $U$ on $[-1,1]$. As $n$ grows, the four points we fixed will not always be 
in the middle between nodes, causing the Lebesgue functions at these points to fluctuate. However, overall, it coincides with our reference line.

This also indicates that the Lebesgue constant of an interpolation method, even if it grows rapidly, does not imply 
that it will be universally affected by rounding errors. For points where the value of the potential function is 
close to the maximum, the impact from rounding errors will be small. Conversely, for points other than the 
interpolation nodes, the lower the value of the potential function, the greater the amplification of the rounding error.

Reducing the Lebesgue constant of polynomial interpolation can be achieved effectively by adding poles. 
The density function in EXAMPLE 1 has a larger potential difference compared to that of the density function 
of equidistant nodes, resulting in a faster exponential growth rate for its Lebesgue constant. However, 
if we add poles (external field), such as the external field given by 
\begin{equation*}
  \phi(x)=\frac{1}{2}\log|x-0.5\mathrm{i}|+\frac{1}{2}\log|x+0.5\mathrm{i}|
\end{equation*}
in the \cref{sec:int}, the difference in the potential function is reduced (EXAMPLE 2), as shown in \cref{fig:4.2a}.

\begin{figure}[htbp]
  \centering
  \subfloat[]{
    \label{fig:4.2a}
    \includegraphics[width=0.47\linewidth]{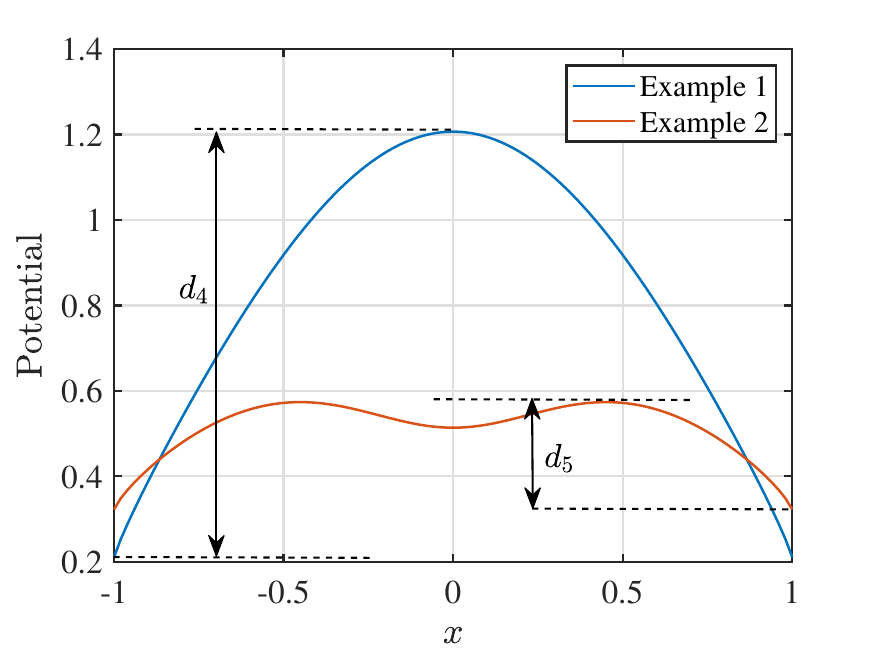}
  }
  \subfloat[]{
    \label{fig:4.2b}
    \includegraphics[width=0.47\linewidth]{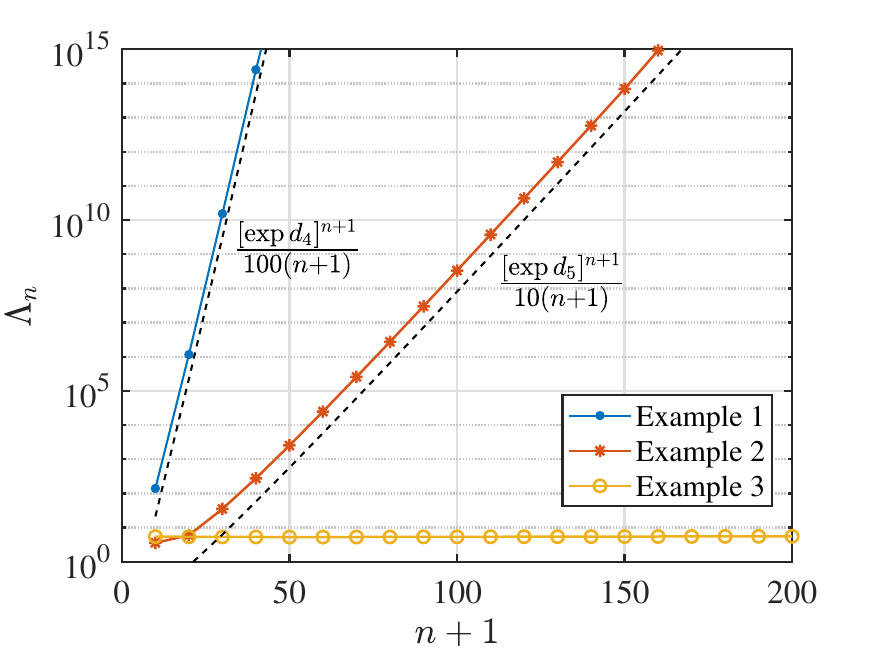}
  }
  \caption{ (a): Potential functions of EXAMPLE 1 and EXAMPLE 2 along with the difference between their respective maximum and minimum values: $d_4$, $d_5$. 
  (b): Lebesgue constants of EXAMPLE 1, EXAMPLE 2, and EXAMPLE 3 with their respective growth reference lines.}
  \label{fig:4.2}
\end{figure}

\cref{th:4.1} shows that the exponentially growing part of the Lebesgue constant for EXAMPLE 1 and EXAMPLE 2 
is $\mathcal{O}([\exp d_4]^n)$ and $\mathcal{O}([\exp d_5]^n)$, respectively. It is straightforward to 
demonstrate that this density function corresponds to the parameter $b_1=1$. Assuming $\delta_n=\mathcal{O}(1/n)$, 
the lower bounds on the growth rate of the Lebesgue constants for both are given by $\mathcal{O}([\exp d_4]^n/n^2)$ and $\mathcal{O}([\exp d_5]^n/n^2)$. 
However, the reference lines provided in the figure are given by $\mathcal{O}([\exp d_4]^n/n)$ and $\mathcal{O}([\exp d_5]^n/n)$, 
which are $n$ times larger than the estimates in \cref{th:4.1}.

The reason for this is that the lower bound in \cref{th:4.1} arises from one basis function $|R_i^{(n)}(\zeta_j^{(n)})|$ 
with a growth rate of $\mathcal{O}(\rho^n/n^{b_1+1})$ where $\rho=\exp(U(\alpha)-U(\beta))$. In reality, there is not just one basis function with a similar exponential growth, 
but an infinite number of them of the same order as $n$. 
Let a very small quantity $c>0$, $\mathbf{X}_c = \{x\in[-1,1]:U(x) > U(\alpha)-c\}$ if $U$ is continuous on $[-1,1]$. 
Then, for any node $x_k^{(n)}\in \mathbf{X}_c$, $|R_k^{(n)}(\zeta_j^{(n)})|$ grows faster than $\mathcal{O}(\rho_c^n/n^{b_1+1})$ where $\rho_c=\exp(U(\alpha)-U(\beta)-c)$. 
The number of these nodes $n_c$ satisfies $\lim_{n\to\infty}{n_c}/{n+1}=\int_{\mathbf{X}_c}w(t)\, \mathrm{d}t$.
Thus the actual growth rate in Examples 1 and 2 is $n$ times the lower bound of \cref{th:4.1}.

\begin{figure}[htbp]
  \centering
  \subfloat[]{
    \label{fig:4.3a}
    \includegraphics[width=0.47\linewidth]{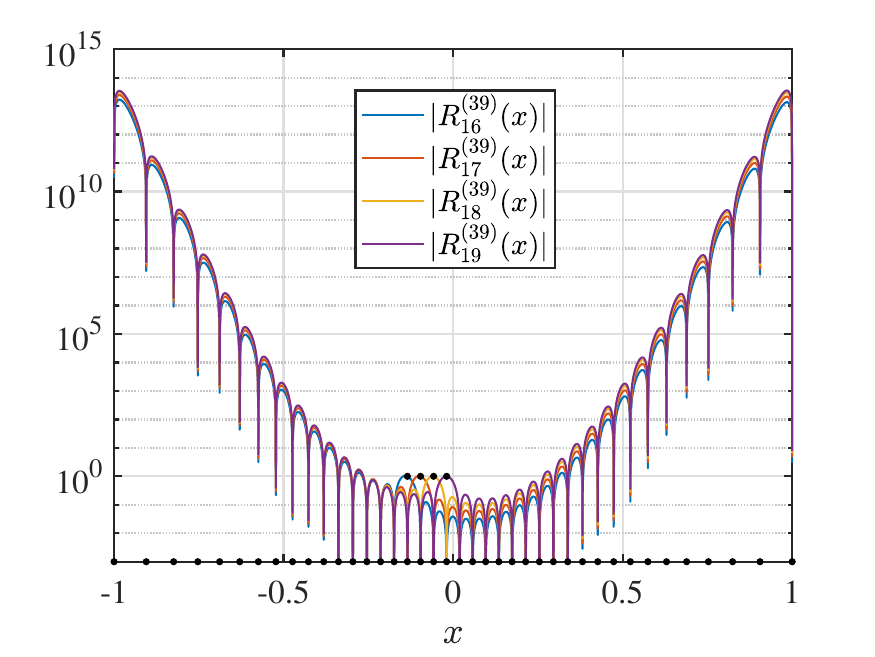}
  }
  \subfloat[]{
    \label{fig:4.3b}
    \includegraphics[width=0.47\linewidth]{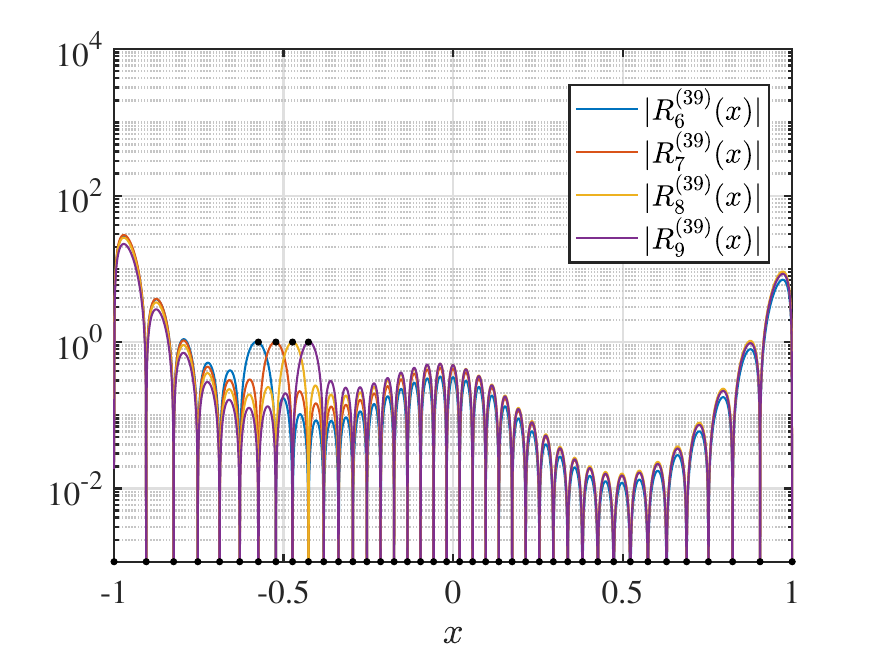}
  }
  \caption{Absolute values of several Lagrangian basis functions for EXAMPLE 1 (a) and EXAMPLE 2 (b) when $n=39$. 
  $R_k^{(n)}$ denotes the basis functions on $x_k^{(n)}$.}
  \label{fig:4.3}
\end{figure}

A question arises: Can the exponential growth of the Lebesgue constant be avoided by adding poles? The answer is yes, 
as long as the potential function can be made constant on $[-1,1]$.
We want $\phi_n$ to satisfy
\begin{equation*}
  \lim_{n\to\infty} \int_{-1}^{1}\log\frac{1}{|x-t|}w(t)\,\mathrm{d}t+\phi_n(x)=\bar{U}, \,\,\forall x\in[-1,1]
\end{equation*}
where $\bar{U}$ is a constant. We can set 
\begin{equation*}
 \phi_n(x)=\bar{U}- \int_{-1}^{1}\log\frac{1}{|x-t|}w(t)\,\mathrm{d}t.
\end{equation*}
Then according Eq. \cref{eq:2.1}, we have
\begin{equation*}
  w_k=C (-1)^k{\exp[-(n+1)\int_{-1}^{1}\log\frac{1}{|x_k-t|}w(t)\,\mathrm{d}t]}\Big/{\prod_{i=0,i\ne k}^n|x_k-x_i|},\,\, C\not=0.
\end{equation*}
We used such a rational interpolant for the density function in EXAMPLE 1 to obtain EXAMPLE 3. The new rational interpolant avoids the exponential growth of Lebesgue constant, as illustrated in \cref{fig:4.2b}.

\section{The Lebesgue constant in an equilibrium potential}\label{sec:5}
In the case of the equilibrium potential, EXAMPLE 3 demonstrates no exponential growth in its Lebesgue constant. 
However, \cref{th:4.1} only establishes that the exponential growth part of the lower bound estimate for the Lebesgue 
constant vanishes for the equilibrium potential. To ascertain whether the Lebesgue constant avoids exponential growth, 
we need to estimate its upper bound.

This proof is divided into two steps. First, it is necessary to demonstrate that the amplitude of the 
oscillations of the basis functions does not increase exponentially with $n$, i.e., the phenomenon depicted 
in \cref{fig:4.3} does not occur (\cref{le:5.1}, \cref{le:5.2}). The second step is then to estimate the 
accumulation of the absolute values of the basis functions to obtain a consistent upper bound for the 
Lebesgue function - in other words, to obtain an upper bound estimate for the Lebesgue constant (\cref{th:5.3}). 
The results show the existence of a non-exponentially increasing upper bound for the Lebesgue constant at 
the equilibrium potential.

\begin{lemma}\label{le:5.1}
  Two right triangles $\Delta ABC$ and $\Delta DEF$ are shown in Figure 1. Extend $AB$ to intersect $EF$ at $H$. Make $HG$ perpendicular to $BE$ at $G$. If $|BC| = |DE|$, then 
  \[
    \frac{|BG|}{|EG|}=\frac{|DF|}{|AC|}.
  \]
\end{lemma}

\begin{figure}[htbp]
  \centering
  \subfloat[]{
    \label{fig:5.1a}
    \includegraphics[width=0.38\linewidth]{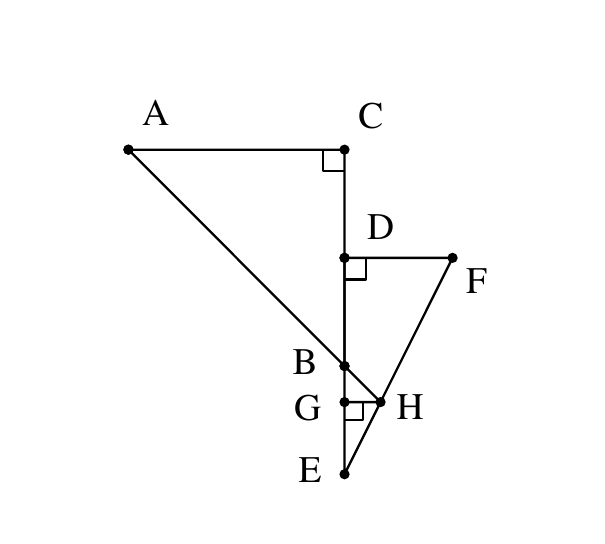}
  }
  \subfloat[]{
    \label{fig:5.1b}
    \includegraphics[width=0.56\linewidth]{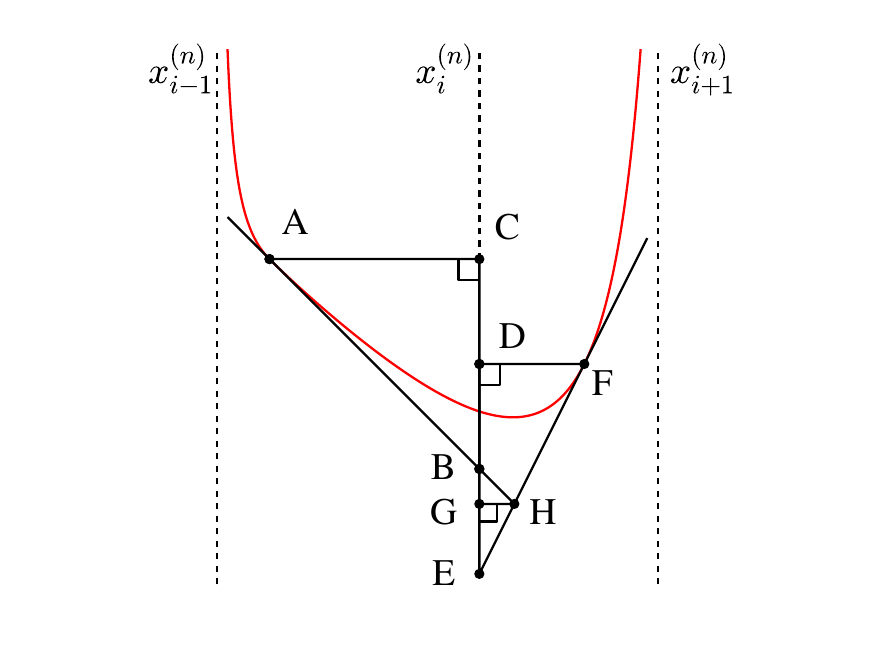}
  }
  \caption{(a): Schematic representation of \cref{le:5.1}. (b): Schematic representation of the proof of \cref{le:5.2}.}
  \label{fig:5.1}
\end{figure}

\begin{lemma}\label{le:5.2}
  Under the assumptions of \cref{th:3.4}, if $U$ is constant and equal to the constant $C$ on [-1,1], then the Lagrangian basis function $R_i^{(n)}$ with rational interpolation $r_{n,m}$ has the following upper bound for a sufficiently large $n$:
\begin{itemize}
  \item For all $x\in(x_{i-1}^{(n)},x_{i+1}^{(n)})$
  \begin{equation*}
    |R_i^{(n)}(x)|\le e^{(n+1)(\delta_n^-+\delta_n^+)+1};
  \end{equation*}
  \item For all $x\in[-1,1]-(x_{i-1}^{(n)},x_{i+1}^{(n)})$
  \begin{equation*}
    |R_i^{(n)}(x)|\le \min(e^{(n+1)(\delta_n^-+\delta_n^+)+1}\frac{|\zeta_{i}^{(n)}-x_i^{(n)}|}{|x-x_i^{(n)}|},\, e^{(n+1)(\delta_n^-+\delta_n^+)+1}\frac{|\zeta_{i+1}^{(n)}-x_i^{(n)}|}{|x-x_i^{(n)}|}).
  \end{equation*}
\end{itemize}
\end{lemma}

\begin{proof}
  According the assumptions, we have $U_n(x)>U(x)-\delta_n^-=C-\delta_n^-$. Let 
  \begin{equation*}
    u_i^{(n)}(x)=U_n(x)+\frac{1}{n+1}\log|x-x_i^{(n)}|.
  \end{equation*}

\begin{itemize}
  \item For all $x\in(x_{i-1}^{(n)},x_{i+1}^{(n)})$.
  
  Since \cref{le:3.1} and \cref{re}, $u_i^{(n)}$ is convex on $(x_{i-1}^{(n)},x_{i+1}^{(n)})$ for a sufficiently large $n$. 
  As show in \cref{fig:5.1b}, let $A=(\zeta_{i}^{(n)},u_i^{(n)}(\zeta_{i}^{(n)}))$, $F=(\zeta_{i+1}^{(n)},u_i^{(n)}(\zeta_{i+1}^{(n)}))$ 
  and make tangents $AH$ and $FH$ to $u_i^{(n)}$ through $A$ and $B$. 
  Then the convex function $u_i^{(n)}$ will be above the point $H$, i.e., 
  \begin{equation*}
    u_i^{(n)}(x)>y_H, \quad \forall x\in(x_{i-1}^{(n)},x_{i+1}^{(n)}).
  \end{equation*}

  Notably 
  \begin{equation*}
    |BC|=|\zeta_{i}^{(n)}-x_i^{(n)}||(u_i^{(n)})'(\zeta_{i}^{(n)})|=\frac{1}{n+1}
  \end{equation*}
  and
  \begin{equation*}
    |DE|=|\zeta_{i+1}^{(n)}-x_i^{(n)}||(u_i^{(n)})'(\zeta_{i+1}^{(n)})|=\frac{1}{n+1}.
  \end{equation*}
  Then from \cref{le:5.1}, we have
  \begin{equation}\label{eq:5.2a}
    u_i^{(n)}(x)>(1-\lambda)y_F+\lambda y_A-\frac{1}{n+1}, \quad \forall x\in(x_{i-1}^{(n)},x_{i+1}^{(n)})
  \end{equation}
  where 
  \begin{equation*}
    \lambda = \frac{|AC|}{|DF|+|AC|},\quad 1-\lambda = \frac{|DF|}{|DF|+|AC|}.
  \end{equation*}

  In addition, 
  \begin{equation}\label{eq:5.2b}
    u_i^{(n)}(x_i^{(n)})\le(1-\lambda)y_A+\lambda y_F.
  \end{equation}
  Considering  
  \begin{equation*}
    C-\delta_n^-+\frac{1}{n+1}\log|AC|<y_A=u_i^{(n)}(\zeta_{i}^{(n)})<C+\delta_n^++\frac{1}{n+1}\log|AC|
  \end{equation*}
  and
  \begin{equation*}
    C-\delta_n^-+\frac{1}{n+1}\log|DF|<y_F=u_i^{(n)}(\zeta_{i+1}^{(n)})<C+\delta_n^++\frac{1}{n+1}\log|DF|,
  \end{equation*}
  it is easy to prove that 
  \begin{equation}\label{eq:5.2c}
    y_F-y_A\le \delta_n^-+\delta_n^+-\frac{1}{n+1}\log\frac{|AC|}{|DF|};\,y_A-y_F\le \delta_n^-+\delta_n^++\frac{1}{n+1}\log\frac{|AC|}{|DF|}.
  \end{equation}

  The basis function $R_i^{(n)}(x)$ has
  \begin{align*}
    R_i^{(n)}(x)&=\dfrac{\prod_{k=0,k\ne i}^{n}|x-x_k^{(n)}|}{\prod_{j=1}^{m}|x-p_j|}|w_i|\\
    &=\exp[-(n+1)u_i^{(n)}(x)]\exp[(n+1)u_i^{(n)}(x_i^{(n)})]\\
    &=\exp[(n+1)(u_i^{(n)}(x_i^{(n)})-u_i^{(n)}(x))].
  \end{align*}
  From Eq. \cref{eq:5.2a}, \cref{eq:5.2b} and Eq. \cref{eq:5.2c}, we have 
  \begin{align*}
    R_k^{(n)}(x)&<\exp\Big[ (n+1)[(1-\lambda)(y_A-y_F)+\lambda(y_F-y_A)]+1\Big]\\
    &\le \exp\Big[\frac{|DF|-|AC|}{|DF|+|AC|}\log\frac{|AC|}{|DF|}+(n+1)(\delta_n^-+\delta_n^+)+1\Big]\\
    &\le \exp[(n+1)(\delta_n^-+\delta_n^+)+1].
  \end{align*}
 
  \item For all $x\in[-1,1]/[x_{i-1}^{(n)},x_{i+1}^{(n)}]$.
  Since $U_n(x)>C-\delta_n^-$, 
  \begin{equation}\label{eq:5.2d}
    u_k^{(n)}(x)>C-\delta_n^-+\frac{1}{n+1}\log|x-x_k^{(n)}|.
  \end{equation}
  It is easy to prove that
  \begin{equation*}
    y_A=u_i^{(n)}(\zeta_{i}^{(n)})\le C+\delta_n^++\frac{1}{n+1}[\log(\lambda)+\log(|AC|+|DF|)];
  \end{equation*}
  \begin{equation*}
    y_F=u_i^{(n)}(\zeta_{i+1}^{(n)})\le C+\delta_n^++\frac{1}{n+1}[\log(1-\lambda)+\log(|AC|+|DF|)].
  \end{equation*}
  Then we have
  \begin{align*}
    u_i^{(n)}(x_i^{(n)}) &\le (1-\lambda)y_A+\lambda y_F \\
    &\le C+\delta_n^++\frac{(1-\lambda)\log(\lambda)+\lambda\log(1-\lambda)+\log(|AC|+|DF|)}{n+1}.
  \end{align*} 

  It is easy to verify that 
  \begin{align*}
    (1-\lambda)\log(\lambda)+\lambda\log(1-\lambda)&<\log(\lambda)+1;\\
    (1-\lambda)\log(\lambda)+\lambda\log(1-\lambda)&<\log(1-\lambda)+1,
  \end{align*}
  then 
  \begin{equation}\label{eq:5.2e}
    u_i^{(n)}(x_i^{(n)})< C+\delta_n^++\frac{\log|AC|+1}{n+1};
  \end{equation}
  \begin{equation}\label{eq:5.2f}
    u_i^{(n)}(x_i^{(n)})< C+\delta_n^++\frac{\log|DF|+1}{n+1}.
  \end{equation}
  Finally, the lemma is proved due to Eq. \cref{eq:5.2d}, \cref{eq:5.2e}, \cref{eq:5.2f}, and 
  \begin{equation*}
    |R_i^{(n)}(x)|=\exp[(n+1)(u_i^{(n)}(x_i)-u_i^{(n)}(x))].
  \end{equation*}

\end{itemize}
\end{proof}

\begin{theorem}\label{th:5.3}
  Under the assumptions of \cref{th:3.4}, if $U$ is constant and equal to the constant $C$ on [-1,1], then the Lebesgue constant $\Lambda_n$ of rational interpolation $r_{n,m}$ has the following upper bound for a sufficiently large $n$:
  \begin{equation*}
    \Lambda_n\le e^{(n+1)(\delta_n^-+\delta_n^+)+1}(5+2b_1\log{a_1}+2b_1\log{n}).
  \end{equation*}
\end{theorem}

\begin{proof}
  Consider the Lebesgue function $\Lambda_n(x)$, where $x\in[-1,1]$.
  If $x\in\{x_k^{(n)}\}_{k=0}^n$, $\Lambda_n(x)=1$. Setting $x\notin \{x_k^{(n)}\}_{k=0}^n$, then we can find $i\in\{1,2,\dots,(n-1)\}$ such that $x\in(x_{i-1}^{(n)},x_{i+1}^{(n)})$.
  From \cref{le:5.2}, we have
  \begin{align*}
    \Lambda_n(x) &=\sum_{k=0}^{n}|R_k^{(n)}(x)|\\
    &\le e^{(n+1)(\delta_n^-+\delta_n^+)+1}[\sum_{k=0}^{i-1}\frac{|\zeta_{k+1}^{(n)}-x_{k}^{(n)}|}{|x-x_{k}^{(n)}|}+1+\sum_{k=i+1}^{n}\frac{|\zeta_{k}^{(n)}-x_{k}^{(n)}|}{|x-x_{k}^{(n)}|}]\\
    &< e^{(n+1)(\delta_n^-+\delta_n^+)+1}[\sum_{k=0}^{i-3}\frac{|\zeta_{k+1}^{(n)}-x_{k}^{(n)}|}{|x-x_{k}^{(n)}|}+5+\sum_{k=i+3}^{n}\frac{|\zeta_{k}^{(n)}-x_{k}^{(n)}|}{|x-x_{k}^{(n)}|}]\\
    &< e^{(n+1)(\delta_n^-+\delta_n^+)+1}[\sum_{k=0}^{i-3}\frac{|x_{k+1}^{(n)}-x_{k}^{(n)}|}{|x-x_{k}^{(n)}|}+5+\sum_{k=i+3}^{n}\frac{|x_{k-1}^{(n)}-x_{k}^{(n)}|}{|x-x_{k}^{(n)}|}]\\
    &< e^{(n+1)(\delta_n^-+\delta_n^+)+1}[\int_{-1}^{x_{i-2}^{(n)}}\frac{1}{x-t}\,\mathrm{d}t+5+\int_{x_{i+2}^{(n)}}^{1}\frac{1}{t-x}\,\mathrm{d}t]\\
    &= e^{(n+1)(\delta_n^-+\delta_n^+)+1}[\log(1-x^2)+5-\log(x-x_{i-2}^{(n)})-\log(x_{i+2}^{(n)}-x)]\\
    &< e^{(n+1)(\delta_n^-+\delta_n^+)+1}[5-\log(x_{i-1}^{(n)}-x_{i-2}^{(n)})-\log(x_{i+2}^{(n)}-x_{i+1}^{(n)})]\\
    &< e^{(n+1)(\delta_n^-+\delta_n^+)+1}[5-2\log(a_1 n^{-b_1})]\\
    &< e^{(n+1)(\delta_n^-+\delta_n^+)+1}[5+2b_1\log{a_1} +2b_1\log{n}]
  \end{align*}
  Therefore $\Lambda_n\le e^{2(n+1)(\delta_n^-+\delta_n^+)+1}(5+2b_1\log{a_1}+2b_1\log{n})$.
\end{proof}

\section{Interpolation Error and Lebesgue Constant}\label{sec:6}
According to the conclusion of the previous section, if a rational interpolation, whose nodes obey some nonzero 
density function $w$, has an external field $\phi_n$ that satisfies 
\begin{equation*}
  \lim_{n\to\infty}\max_{x\in[-1,1]}|\hat{U}_n(x)-\bar{U}|=0, \quad \hat{U}_n(x):= \int_{-1}^1\log\frac{1}{|x-t|}w(t)\,\mathrm{d}t+\phi_n(x)
\end{equation*}
where $\bar{U}$ is a constant.
Then the Lebesgue constant of the rational interpolation does not grow exponentially. 

However, the condition to avoid the exponential growth of the Lebesgue constant is not sufficient to achieve 
exponential convergence of the error. An additional condition needed to make rational nterpolation exponentially 
convergent for analytic functions is that the external field $\phi_n$ converges consistently on a neighborhood of $[-1,1]$. 
As shown in \cref{fig:6.1}, if the external field converges consistently on a neighborhood of $[-1,1]$, 
then there exists an enclosing path $\Gamma$ such that the interior of the enclosing path does not contain poles as 
well as singularities of the interpolated function. Then the exponential decay of the error is obtained according 
to the Hermite integral formula for rational interpolation \cite{Zhao2023}.
\begin{figure}[htbp]
	\centering
	\includegraphics[width=0.95\linewidth]{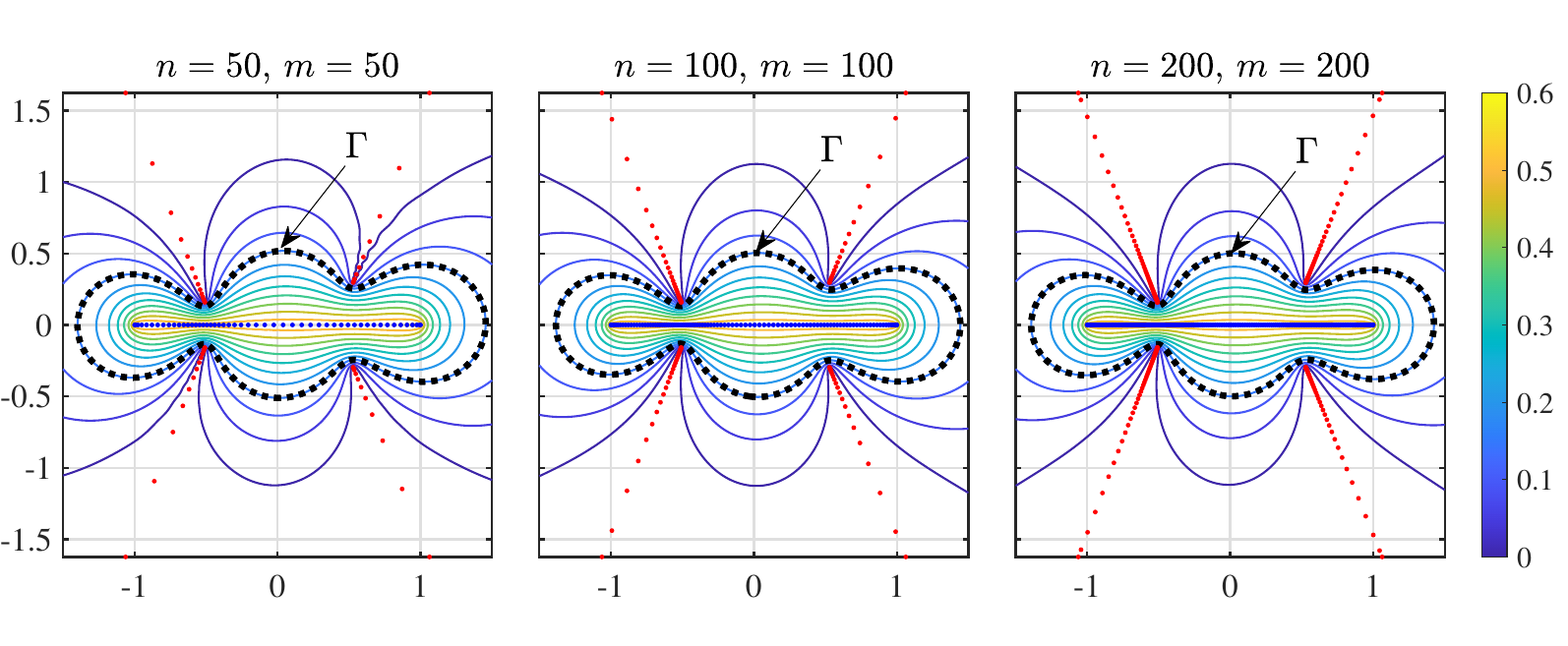}
  \caption{Contours of the potential function $U_n$ generated by $(n+1)$ nodes (blue) and $m$ poles (red). 
  As $n$ grows, the potential function stabilizes in some neighborhood of $[-1,1]$. The computation of these nodes 
  and poles is detailed in \cite{Zhao2023}.}
  \label{fig:6.1}
\end{figure}


\subsection{Take the Floater-Hormann interpolation as an example}\label{sebsec:6.1}
Barycentric rational interpolation is often characterized by the barycentric weights and nodes \cite{Schneider1986,Floater2007,Hale2009}. 
In these methods, the poles of the rational interpolation are implicit but can be uniquely determined from the 
barycentric weights and nodes. One way to compute the poles using weights and nodes is to translate them into the 
computation of matrix eigenvalues, as detailed in \cite[Chapter 2.3.3]{Klein2012}. We still denote these poles as $\{p_j^{(n)}\}_{j=1}^{m}$, 
which gives us the external field $\phi_n$.

F-H rational interpolation is one of the very effective methods of rational interpolation \cite{Floater2007,Guttel2012,Klein2013}.
There are an F-H rational interpolation that achieves exponential convergence to analytic functions at equidistant nodes \cite{Guttel2012}. 
The barycentric weights of the equidistant nodes are denoted as 
\begin{equation}\label{eq:6.1}
  w_i^{(n,d)}=(-1)^i\sum_{k=d}^{n}\binom{d}{k-i},\quad \frac{d(n)}{n}\to C_{FH}\,(n\to\infty)
\end{equation}
for a fixed $C_{FH}\in(0,1]$.
\cref{fig:6.4} illustrates the contours of the discrete potential for the equidistant nodes at $C_{FH} = 0.25$. 
As $n$ increases, the external field generated by the poles remains stable in a neighborhood on $[-1,1]$. 
Thus, by directly observing the potential function, the method may achieve exponential convergence to the 
analytic function.

\begin{figure}[htbp]
	\centering
	\includegraphics[width=0.95\linewidth]{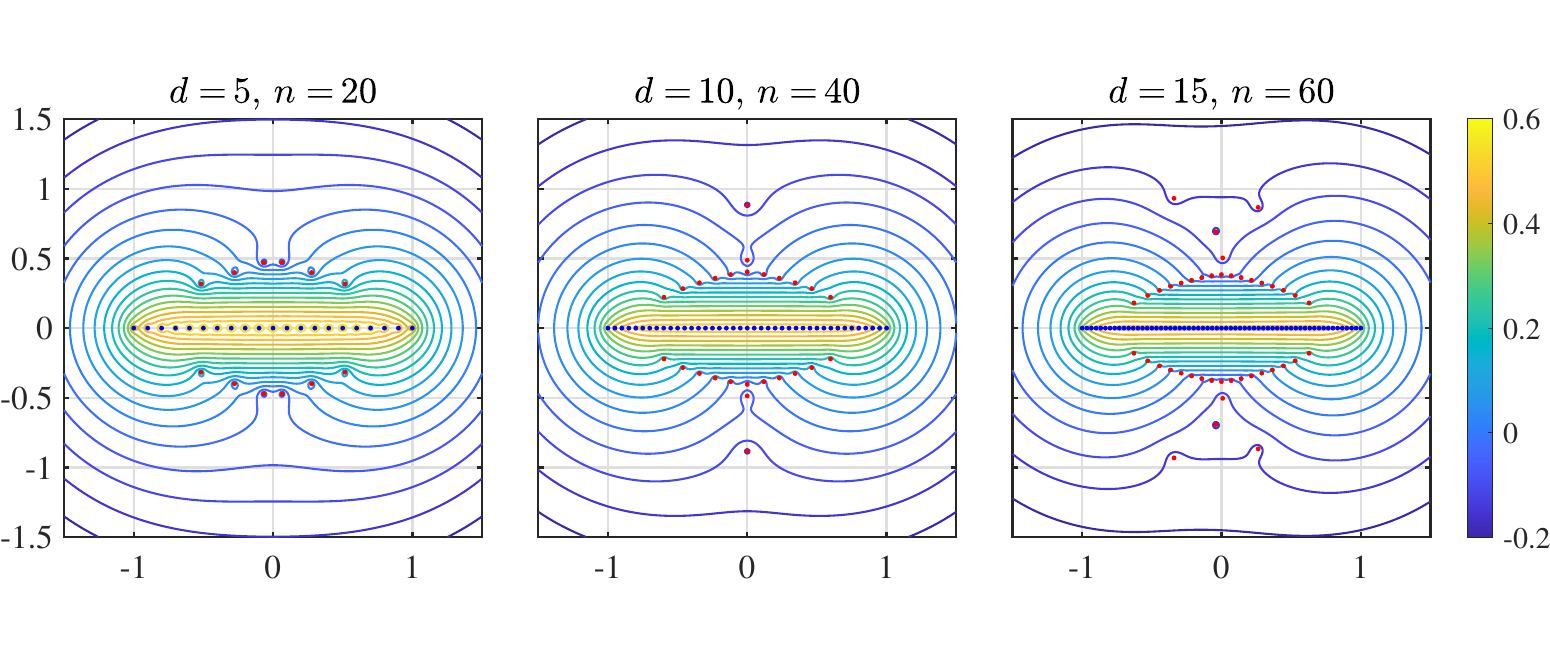}
  \caption{Contours of the potential function $U_n$ generated by the nodes (blue) and poles (red) of the F-H rational interpolation [] when $d/n=0.25$. 
  As $n$ grows, the potential function stabilizes in some neighborhood of $[-1,1]$.}
  \label{fig:6.4}
\end{figure}

However, \cref{fig:6.4} shows that the potential function is not equalized on $[-1,1]$, and the potential is 
lower at the edges of the interval. The reason for this is that the potential generated by the uniform density 
is weakly singular (with an unbounded derivative) at both endpoints. In contrast, the potential generated by 
the poles converges on a neighborhood of $[-1,1]$, so its derivative on $[-1,1]$ is bounded. As a result, 
the potential function will not be a constant on $[-1,1]$. 

According to the conclusion in \cref{sec:4}, the Lebesgue constant has an exponential growth rate of $(\exp\rho)^n$, 
where $\rho$ is the difference between the maximum and minimum values of the potential function on $[-1,1]$. 
The \cref{sec:3} shows that the growth rate of the barycentric weight ratio is likewise $(\exp\rho)^n$. 
Thus, from the barycentric weight Eq. \cref{eq:6.1}, the maximum ratio is $2^d\approx (2^{C_{FH}})^n$, and then the exponential 
growth rate of the Lebesgue constant can be obtained to be $(2^{C_{FH}})^n$. This is in agreement with the conclusions of \cite{Klein2013}.

While for F-H rational interpolation with a given parameter $d$ at equidistant nodes, its Lebesgue constant only grows logarithmically \cite{Bos2012}. 
Numerical tests reveal that the difference of its potential function on $[-1,1]$ tends to $0$. \cref{fig:6.2a} shows 
the potential function $\hat{U}_n(x)$ for F-H rational interpolation on equidistant nodes for $d=4$. As $n$ grows, the potential 
function gets closer to a constant function on $[-1,1]$. For different parameters $d$, \cref{fig:6.2b} shows the 
difference between the maximum and minimum values of the potential function on $[-1,1]$.

\begin{figure}[htbp]
  \centering
  \subfloat[]{
    \label{fig:6.2a}
    \includegraphics[width=0.47\linewidth]{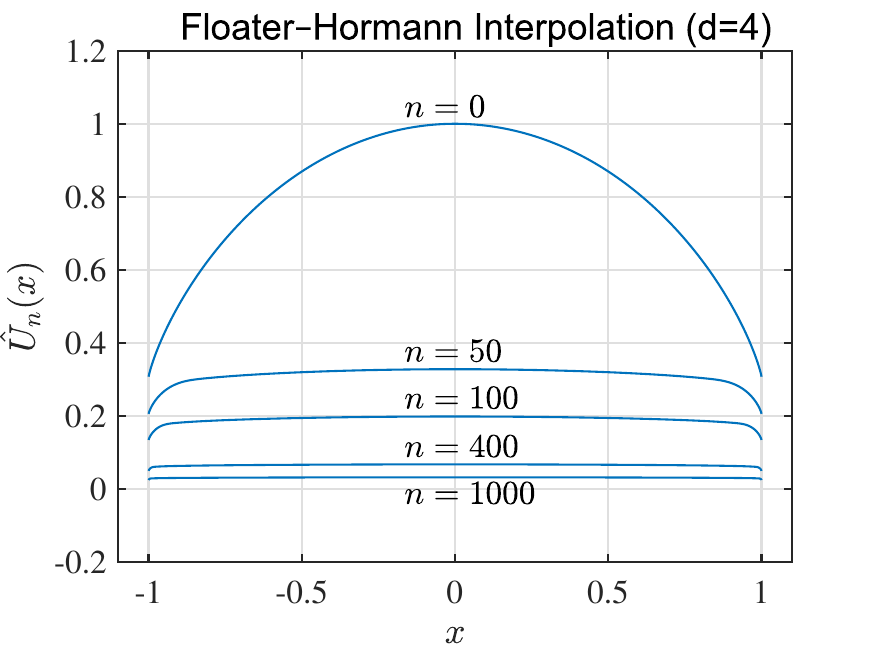}
  }
  \subfloat[]{
    \label{fig:6.2b}
    \includegraphics[width=0.47\linewidth]{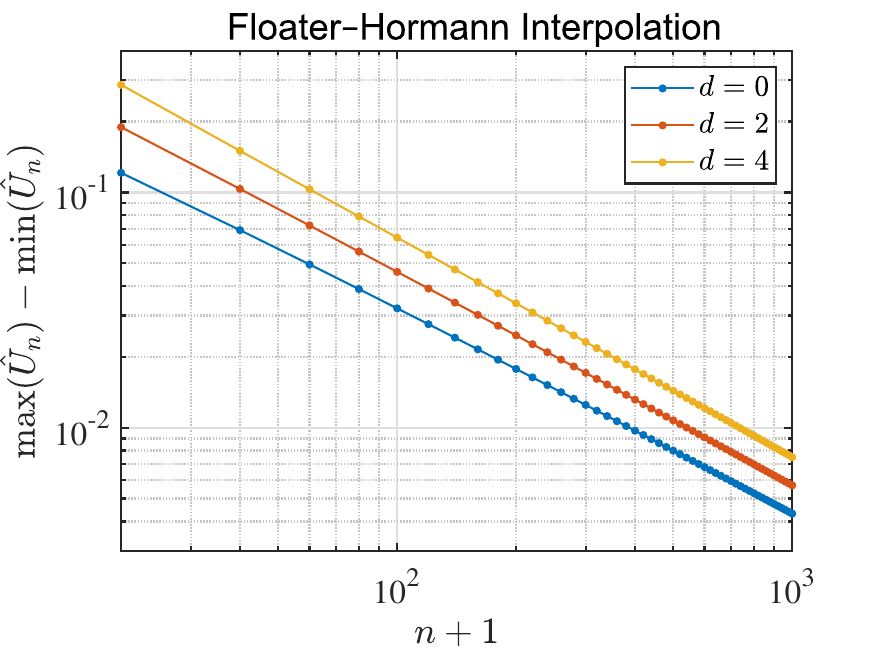}
  }
  \caption{(a): Potential function $\hat{U}_n$ corresponding to different $n$ of F-H rational interpolation when $d=4$. 
  (b): Difference between the maximum and minimum values of potential function $\hat{U}_n$ on $[-1,1]$ for F-H rational interpolation when $d=0,2,4$ respectively.}
  \label{fig:6.2}
\end{figure}

Since the difference of the potential functions converges to $0$, the external field $\phi_n(x)$ converges uniformly 
to $C_n-U_a$ on $[-1,1]$, where $C_n$ is a constant. Due to the weak singularity of $U_a$, it is easy to show that the 
derivative of $\phi_n$ on $[-1,1]$ has no upper bound as $n$ grows. Considering the definition of the external 
field $\phi_n$, it is clear that the poles will grow closer to the real axis as $n$ grows. As shown in \cref{fig:6.5}, 
unlike the case in \cref{fig:6.1}, there is no integral path $\Gamma$ to provide a sufficient difference in potential. 
Thus, the F-H rational interpolation on equidistant nodes has no exponential convergence to analytic functions 
when parameter $d$ is fixed.

\begin{figure}[htbp]
	\centering
	\includegraphics[width=0.95\linewidth]{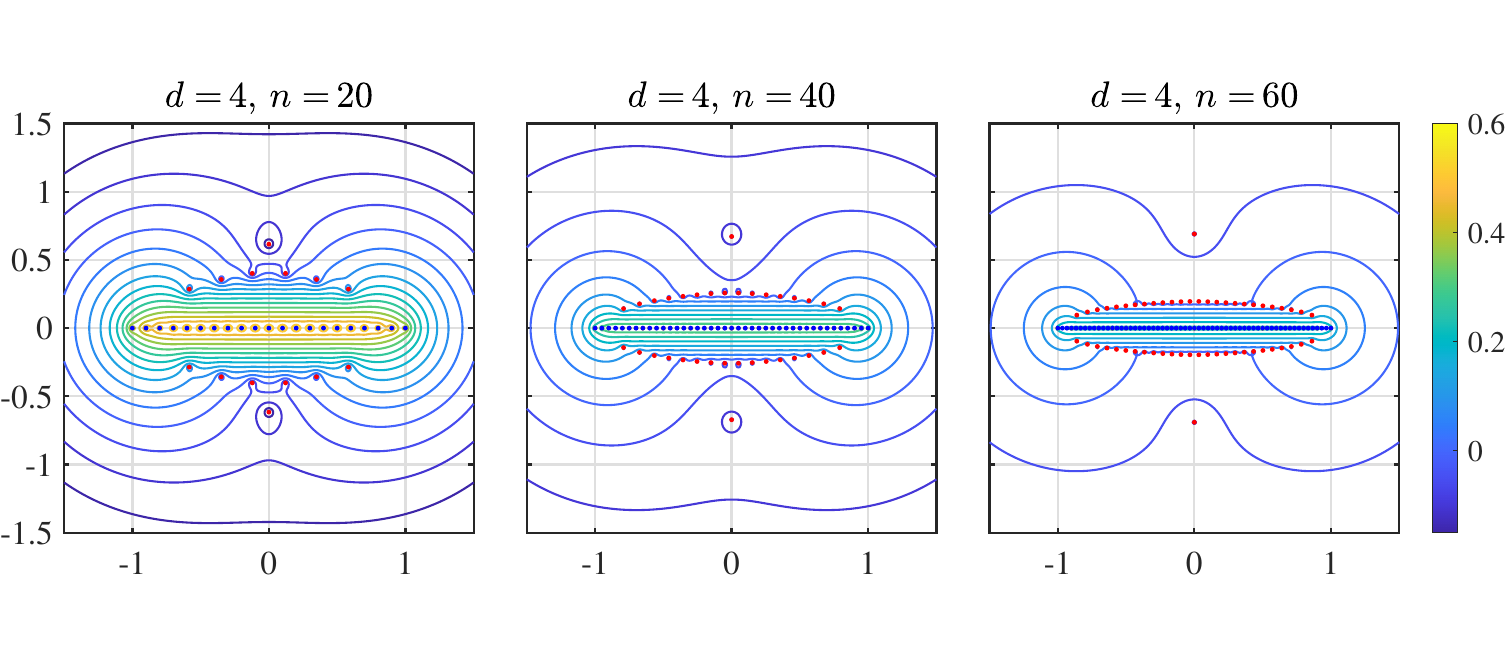}
  \caption{Contours of the potential function $U_n$ generated by the nodes (blue) and poles (red) of the F-H rational interpolation \cite{Floater2007} when $d=4$. 
  As $n$ grows, the poles get closer to the interval $[-1,1]$ and the potential function takes on smaller values on the neighborhood of $[-1,1]$.}
  \label{fig:6.5}
\end{figure}

\subsection{Fast and stable rational interpolation}
\cref{sebsec:6.1} supports the well-established view that rational interpolation on equidistant nodes does not 
exhibit stable exponential convergence to the analytic function. From a potential perspective, the key reason for 
this is that the density function of equidistant nodes produces a logarithmic potential function $U_a$ 
that is weakly singular on $[-1,1]$. This makes it difficult to match the external fields generated by the poles.

Conversely, if a positive density function $w$ produces a potential that is analytic on $[-1,1]$ and can be analytically 
extended to some neighborhood $\Omega$ of $[-1,1]$. If there exists a harmonic function $\phi$ on some neighborhood 
$\Omega'\subseteq\Omega$ such that the potential function is constant on $[-1,1]$ and there exists a family of poles 
$\{p_j^{(n)}\}_{j=1}^m$ (where $m\le n$), the discrete external field $\phi_n$ generated by $\{p_j^{(n)}\}_{j=1}^m$ 
converges consistently to $\phi$ on $\Omega'$. Then the rational interpolation formed by this density function $w$ 
with the poles $\{p_j^{(n)}\}_{j=1}^m$ both avoids the exponential growth of Lebesgue constant and achieves 
exponential convergence for the analytic function.

\section{Conclusions}
\label{sec:conclusions}
In this work, we focus on the connection between the continuous potential function $U$, the barycentric weights, 
and Lebesgue constant for polynomial and rational interpolation on the interval $[-1,1]$. The difference $d$ between 
the maximum and minimum of the potential function induces an exponential growth rate lower bound $[\exp d]^{n}$ for the 
maximum ratio of the absolute values of the barycentric weights, which is also reflected in the Lebesgue constant. 
Moreover, we find that the exponential growth rate of the Lebesgue function at $\hat{x}\in[-1,1]$ can be approximately 
estimated using the potential function. Finally, when the potential function equilibrates on $[-1,1]$, we give a 
non-exponential growth upper bound on its Leberger constant.

It is important to note that the Lebesgue constant is not entirely determined by the continuous potential \(U\). 
For example, Chebyshev polynomial interpolation and Legendre polynomial interpolation, which have the same continuous 
potential function, have different growth rates of Lebesgue constant. In this paper, the potential function \(U\) 
is derived from the density function \(w\) and the external field \(\phi\), while the Lebesgue constant is influenced 
by the interpolation nodes and poles. The bridge connecting these two aspects is the discrete potential \(U_n\), defined 
by the nodes and poles. The manner in which the discrete potential converges to the continuous potential also impacts 
the growth of the Lebesgue constant. To characterize the convergence of the discrete potential to the continuous 
potential, we introduce the parameters \(\delta_n^{\pm }\).
\begin{figure}[htbp]
	\centering
	\includegraphics[width=0.95\linewidth]{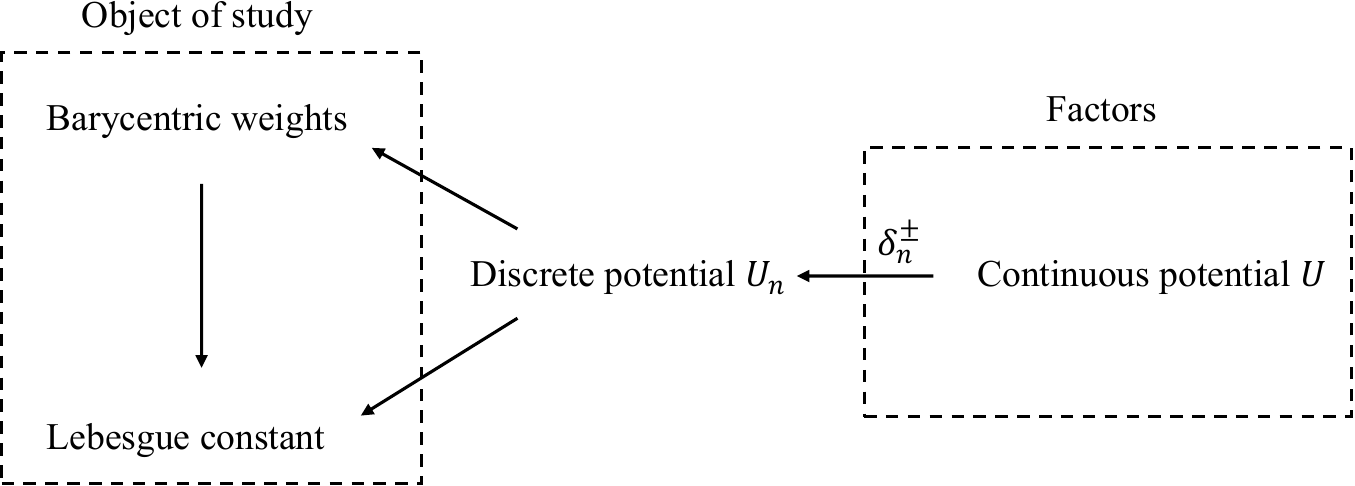}
  \caption{The main research path of this paper.}
  \label{fig:7}
\end{figure}

Since the focus of this paper is on the connection between the continuous potential and the Lebesgue constant, 
we deliberately ignore the discussion on \(\delta_n^{-}\) and \(\delta_n^{+}\). In the numerical examples, all node 
distributions satisfy $\int_{-1}^{x_i^{(n)}}w(t)\,\mathrm{d}t=i/n$, and all employ a fixed \(\phi_n\). The purpose of these special arrangements is 
to minimize the effect on Lebesgue constant of the process by which the discrete potential converges to the continuous 
potential. In our tests, these special arrangements make \(\delta_n^{-} = \mathcal{O}(n^{-1})\) and 
\(\delta_n^{+} = \mathcal{O}(n^{-1}\log n)\). Thus, the numerical correlation between the continuous potential and 
Lebesgue constant becomes clearer.

Although we ignore \(\delta_n^{\pm}\) in this paper, it may be the key to explaining some phenomena. 
Examples include different Lebesgue constants for various Jacobi polynomial interpolations. We believe that discussing 
the connection between specific node distributions and \(\delta_n^{\pm}\) is also an interesting issue that 
may provide new approaches to the study of Lebesgue constants for a specific polynomial or rational interpolation.



\bibliographystyle{siamplain}
\bibliography{reference}
\end{document}

%% file: ex_shared.tex
\usepackage{lipsum}
\usepackage{bm}
\usepackage{amsfonts}
\usepackage{graphicx}
\usepackage[caption=false,farskip=0pt]{subfig}
\usepackage{epstopdf}
\usepackage{algorithmic}

\ifpdf
  \DeclareGraphicsExtensions{.eps,.pdf,.png,.jpg}
\else
  \DeclareGraphicsExtensions{.eps}
\fi

\newsiamremark{remark}{Remark}
\newsiamremark{hypothesis}{Hypothesis}
\crefname{hypothesis}{Hypothesis}{Hypotheses}
\newsiamthm{claim}{Claim}

\headers{potential, barycentric weights, and Lebesgue constants}{Kelong Zhao, and Shuhuang Xiang}

\title{Polynomial and rational interpolation: potential, barycentric weights, and Lebesgue constants\thanks{Submitted to the editors DATE.
\funding{This work was funded by National Science Foundation of China (No. 12271528).}}}

\author{Kelong Zhao\thanks{School of Mathematics and Statistics, Central South University, Changsha 410083, Hunan, People's Republic of China 
  (\email{clonezhao.1994@csu.edu.cn}).}
\and Shuhuang Xiang\thanks{School of Mathematics and Statistics, Central South University, Changsha 410083, Hunan, People's Republic of China 
  (\email{xiangsh@csu.edu.cn}) (corresponding author).}}

\usepackage{amsopn}


%% file: LC_EP_v1.bbl
\begin{thebibliography}{10}

\bibitem{Abramowitz1964}
{\sc M.~Abramowitz and I.~A. Stegun}, {\em Handbook of mathematical functions
  with formulas, graphs, and mathematical tables}, vol.~No. 55 of National
  Bureau of Standards Applied Mathematics Series, U. S. Government Printing
  Office, Washington, DC, 1964.
\newblock For sale by the Superintendent of Documents.

\bibitem{Berrut1997}
{\sc J.-P. Berrut}, {\em The barycentric weights of rational interpolation with
  prescribed poles}, J. Comput. Appl. Math., 86 (1997), pp.~45--52,
  \url{https://doi.org/10.1016/S0377-0427(97)00147-7}.

\bibitem{Bos2011}
{\sc L.~Bos, S.~De~Marchi, and K.~Hormann}, {\em On the {L}ebesgue constant of
  {B}errut's rational interpolant at equidistant nodes}, J. Comput. Appl.
  Math., 236 (2011), pp.~504--510,
  \url{https://doi.org/10.1016/j.cam.2011.04.004}.

\bibitem{Bos2012}
{\sc L.~Bos, S.~De~Marchi, K.~Hormann, and G.~Klein}, {\em On the {L}ebesgue
  constant of barycentric rational interpolation at equidistant nodes}, Numer.
  Math., 121 (2012), pp.~461--471,
  \url{https://doi.org/10.1007/s00211-011-0442-8}.

\bibitem{Bos2013}
{\sc L.~Bos, S.~De~Marchi, K.~Hormann, and J.~Sidon}, {\em Bounding the
  {L}ebesgue constant for {B}errut's rational interpolant at general nodes}, J.
  Approx. Theory, 169 (2013), pp.~7--22,
  \url{https://doi.org/10.1016/j.jat.2013.01.004}.

\bibitem{Brutman1978}
{\sc L.~Brutman}, {\em On the {L}ebesgue function for polynomial
  interpolation}, SIAM J. Numer. Anal., 15 (1978), pp.~694--704,
  \url{https://doi.org/10.1137/0715046}.

\bibitem{Brutman1997}
{\sc L.~Brutman}, {\em Lebesgue functions for polynomial interpolation---a
  survey}, vol.~4, 1997, pp.~111--127.
\newblock The heritage of P. L. Chebyshev: a Festschrift in honor of the 70th
  birthday of T. J. Rivlin.

\bibitem{Erdos1961}
{\sc P.~Erd\H{o}s}, {\em Problems and results on the theory of interpolation.
  {II}}, Acta Math. Acad. Sci. Hungar., 12 (1961), pp.~235--244,
  \url{https://doi.org/10.1007/BF02066686}.

\bibitem{Floater2007}
{\sc M.~S. Floater and K.~Hormann}, {\em Barycentric rational interpolation
  with no poles and high rates of approximation}, Numer. Math., 107 (2007),
  pp.~315--331, \url{https://doi.org/10.1007/s00211-007-0093-y}.

\bibitem{Guttel2012}
{\sc S.~G\"{u}ttel and G.~Klein}, {\em Convergence of linear barycentric
  rational interpolation for analytic functions}, SIAM J. Numer. Anal., 50
  (2012), pp.~2560--2580, \url{https://doi.org/10.1137/120864787}.

\bibitem{Hale2009}
{\sc N.~Hale and T.~W. Tee}, {\em Conformal maps to multiply slit domains and
  applications}, SIAM J. Sci. Comput., 31 (2009), pp.~3195--3215,
  \url{https://doi.org/10.1137/080738325}.

\bibitem{Klein2012}
{\sc G.~Klein}, {\em Applications of Linear Barycentric Rational
  Interpolation}, thesis, 2012.

\bibitem{Klein2013}
{\sc G.~Klein}, {\em An extension of the {F}loater-{H}ormann family of
  barycentric rational interpolants}, Math. Comp., 82 (2013), pp.~2273--2292,
  \url{https://doi.org/10.1090/S0025-5718-2013-02688-9}.

\bibitem{Schneider1986}
{\sc C.~Schneider and W.~Werner}, {\em Some new aspects of rational
  interpolation}, Math. Comp., 47 (1986), pp.~285--299,
  \url{https://doi.org/10.2307/2008095}.

\bibitem{Szego1939}
{\sc G.~Szeg\"{o}}, {\em Orthogonal {P}olynomials}, vol.~Vol. 23 of American
  Mathematical Society Colloquium Publications, American Mathematical Society,
  New York, 1939.

\bibitem{Trefethen1991}
{\sc L.~Trefethen and J.~Weideman}, {\em Two results on polynomial
  interpolation in equally spaced points}, Journal of Approximation Theory, 65
  (1991), pp.~247--260,
  \url{https://doi.org/https://doi.org/10.1016/0021-9045(91)90090-W}.

\bibitem{Zhao2023}
{\sc K.~Zhao and S.~Xiang}, {\em {Barycentric interpolation based on
  equilibrium potential}}, arXiv e-prints,  (2023), arXiv:2303.15222,
  \url{https://doi.org/10.48550/arXiv.2303.15222}.

\end{thebibliography}
